\documentclass{amsart}
\usepackage{amsfonts}

\theoremstyle{plain}

\newtheorem{corollary}{Corollary}

\newtheorem{lemma}{Lemma}

\newtheorem{remark}{Remark}

\numberwithin{equation}{section}
\input{tcilatex}

\begin{document}
\title[Cantor distribution]{On moments of Cantor and related distributions. }
\author{Pawe\l\ J. Szab\l owski}
\address{Department of Mathematics and Information Sciences,\\
Warsaw University of Technology\\
ul Koszykowa 75, 00-662 Warsaw, Poland}
\email{pawel.szablowski@gmail.com}
\date{September 01, 2013}

\begin{abstract}
We provide several simple recursive formulae for the moment sequence of
infinite Bernoulli convolution. We relate moments of one infinite Bernoulli
convolution with others having different but related parameters. We give
examples relating Euler numbers to the moments of infinite Bernoulli
convolutions. One of the examples provides moment interpretation of Pell
numbers as well as new identities satisfied by Pell and Lucas numbers.
\end{abstract}

\maketitle

\section{Introduction}

The aim of this note is to add a few simple observations to the analysis of
the distribution of the so called fatigue symmetric walk (term appearing in 
\cite{Moris98}). These observations are based on the reformulation of known
results scattered through literature. We however pay more attention to the
moment sequences and less to the properties of distributions that produce
these moment sequences. It seems that the main novelty of the paper lies in
the probabilistic interpretation of Pell and Lucas numbers and easy proofs
of some identities satisfied by these numbers. However in order to place
these results in the proper context we recall definition and basic
properties of infinite Bernoulli convolutions. In deriving properties of
these convolutions we recall some known, important results.

The paper is organized as follows. After recalling definition and basic
facts on the fatigue random walks we concentrate on the moment sequences of
infinite Bernoulli convolutions. We formulate a corollary of the results of
the paper expressed in terms of moment sequence. This corollary formulated
in terms of number sequences provides identities of Pell and Lucas numbers
of even order (Remark \ref{Pell}).

\section{Infinite Bernoulli Convolutions}

Let $\left\{ X_{n}\right\} _{n\geq 1}$ be the sequence of i.i.d. random
variables such that $P(X_{1}\allowbreak =\allowbreak 1)\allowbreak
=\allowbreak P(X_{1}\allowbreak =\allowbreak -1)\allowbreak =\allowbreak
1/2. $ Further let $\left\{ c_{n}\right\} _{n\geq 1}$ be a sequence of reals
such that $\sum_{n\geq 1}c_{n}^{2}<\infty .$ We define random variable: 
\begin{equation*}
S\allowbreak =\allowbreak \sum_{n\geq 1}c_{n}X_{n}.
\end{equation*}%
By Kolmogorov 3 series theorem $S$ exists and moreover it is square
integrable. $ES^{2}\allowbreak =\allowbreak \sum_{n\geq 1}c_{n}^{2}.$
Obviously $ES\allowbreak =\allowbreak 0.$ Let $\varphi (t)$ denote
characteristic function of $S.$ By the standard argument we have $\varphi
(t)\allowbreak =\allowbreak E\exp (it\sum_{s\geq 1}c_{n}X_{n})\allowbreak
=\allowbreak E\prod_{n\geq 1}\exp (itc_{n}X_{n})\allowbreak =\allowbreak
\prod_{n\geq 1}(\exp (itc_{n})/2+\exp (-itc_{n})/2)\allowbreak =\allowbreak
\prod_{n\geq 1}\cos (tc_{n}).$

We will concentrate on the special form of the sequence $c_{n}$ namely we
will assume that $c_{n}\allowbreak =\allowbreak \lambda ^{-n}$ for some $%
\lambda >1.$

It is known (see \cite{JessWint35}) that for all $\lambda $ distribution of $%
S\allowbreak =\allowbreak S\left( \lambda \right) $ is continuous that is $%
P_{S}(\left\{ x\right\} )\allowbreak =\allowbreak 0$ for all $x\in \mathbb{R}%
.$ Moreover it is also known (see \cite{Solom95}, \cite{Solom96}) that if
for almost $\lambda \in (1,2]$ this distribution is absolutely continuous
and for almost all $\lambda \in (1,\sqrt{2}]$ it has square integrable
density. Garsia in \cite{Garsia62}, Theorem 1.8 showed examples of such $%
\lambda $ leading to absolutely continuous distribution. Namely such $%
\lambda \in (1,2)$ are the roots of monic polynomials $P$ with integer
coefficients such that $\left\vert P(0)\right\vert \allowbreak =\allowbreak
2 $ and $\lambda \prod_{\left\vert \alpha _{i}\right\vert >1}\left\vert
\alpha _{i}\right\vert \allowbreak =\allowbreak 2,$ where $\left\{ \alpha
_{i}\right\} $ are the remaining roots of $P$.

There are known (see \cite{Erdos39}, \cite{Erdos40}) countable instances of $%
\lambda \in (1,2]$ that this distribution is singular. We will denote by $%
\varphi _{\lambda }$ the characteristic function of $S\left( \lambda \right)
.$ Following \cite{Erdos39} we know that the values $\lambda $ such that $%
\varphi _{\lambda }(t)$ does not tend to zero as $t\longrightarrow \infty $
consequently related distribution is singular (by Riemann--Lebesgue Lemma)
are the so called Pisot or PV- numbers i.e. sole roots of such monic
irreducible polynomials $P$ with integer coefficients having the property
that all other roots have absolute values less than $1.$ We must then have $%
P(0)\allowbreak =\allowbreak 1.$ Examples of such numbers are the so called
'golden ratio' $\left( 1+\sqrt{5}\right) /2$ or the so called 'silver ratio' 
$1+\sqrt{2}.$ Moreover following \cite{Salem44} one knows that PV numbers
are the only numbers $\lambda \in (1,2]$ for which $\varphi _{\lambda }$
does not tend to zero. Of course singularity of the distribution of $S\left(
\lambda \right) $ can occur for $\lambda $ not being PV numbers.

For $\lambda >2$ it is known that the distribution of $S$ is singular \cite%
{KershWint35}.

To simplify notation we will write $\limfunc{supp}X,$ where $X$ is a random
variable meaning $\limfunc{supp}P_{X},$ where $P_{X}$ denotes distribution
of $X.$ Similarly $X\ast Y$ denotes random variable whose distribution is a
convolution of distributions of $X$ and $Y.$

We have simple Lemma.

\begin{lemma}
\label{_1}i) $\limfunc{supp}(S\left( \lambda \right) )\subset \left[ -\frac{1%
}{\lambda -1},\frac{1}{\lambda -1}\right] .$

In particular:

ia) if $\lambda \allowbreak =\allowbreak 2$ then $S\allowbreak \sim
\allowbreak U([-1,1])$ and

ib) if $\lambda \allowbreak =\allowbreak 3$ then $\limfunc{supp}(S+1/2)$ is
equal to Cantor set.

In general if $\lambda \allowbreak $ is a positive integer then $\limfunc{%
supp}(S+1/(\lambda -1))$ consist of all numbers of the form $\sum_{j\geq
1}r_{j}\lambda ^{-j}$ where $r_{j}\in \{0,2\}$. Moreover the distribution of 
$(S(\lambda )+1/(\lambda -1))$ is 'uniform' on this a set.

ii) $\forall k\geq 1:$%
\begin{equation}
\varphi _{\lambda }(\lambda ^{k}t)\allowbreak =\allowbreak \varphi _{\lambda
}(t)\prod_{j=0}^{k-1}\cos \left( \lambda ^{j}t\right) .  \label{basic}
\end{equation}

iii) $\forall k\geq 1:S\left( \lambda \right) \allowbreak \allowbreak \sim
\allowbreak \allowbreak \sum_{i=1}^{k}\lambda ^{i-1}S_{i}(\lambda ^{k}),$
where $S_{i}\left( \tau \right) $ ( $i\allowbreak \allowbreak =\allowbreak
1,\ldots ,k)\allowbreak $ are i.i.d. random variables each having
distribution $S\left( \tau \right) .$ Consequently $\varphi _{\lambda
}(t)\allowbreak =\allowbreak \prod_{j=1}^{k}\varphi _{\lambda ^{k}}\left(
\lambda ^{j-1}t\right) $

iv) Let us denote $m_{n}(\lambda )\allowbreak =\allowbreak ES(\lambda )^{n}.$
Then $\forall n\geq 1:m_{2n-1}(\lambda )\allowbreak =\allowbreak 0$ and 
\begin{equation*}
m_{2n}(\lambda )\allowbreak =\allowbreak \frac{1}{\lambda ^{2kn}-1}%
\sum_{j=0}^{n-1}\binom{2n}{2j}m_{2j}(\lambda )W_{2(n-j)}^{(k)}(\lambda ),
\end{equation*}%
with $m_{0}\allowbreak =\allowbreak 1$, where $W_{n}^{(1)}\allowbreak
=\allowbreak 1,$ $W_{n}^{(k)}\left( \lambda \right) \allowbreak =\allowbreak
\left. \frac{d^{n}}{dt^{n}}(\prod_{j=1}^{k}\cosh (\lambda
^{j-1}t))\right\vert _{t=0}\allowbreak $\newline
$=\allowbreak \frac{1}{2^{k-1}}\sum_{i_{1}=0,\ldots
,i_{k-1}=0}^{1}(1+\sum_{j=1}^{k}(2i_{j}-1)\lambda ^{j})^{2n}.$ \newline
In particular we have:%
\begin{eqnarray}
m_{2n}(\lambda ) &=&\frac{1}{\lambda ^{2n}-1}\sum_{j=0}^{n-1}m_{2j}(\lambda )%
\binom{2n}{2j},  \label{bezp} \\
m_{2n}\left( \lambda \right)  &=&\frac{1}{\lambda ^{4n}-1}\sum_{j=0}^{n-1}%
\binom{2n}{2j}m_{2j}(\lambda )\sum_{l=0}^{2(n-j)}\binom{2(n-j)}{2l}\lambda
^{2l}.  \label{l^4}
\end{eqnarray}

v) $\forall k\geq 1:m_{2k}(\lambda )\allowbreak \mathbb{=}\allowbreak \frac{%
-1}{\lambda ^{2k}-1}\sum_{j=0}^{k-1}\binom{2k}{2j}\lambda
^{2j}E_{2(k-j)}m_{2j}(\lambda ),$ where $E_{k}$ denotes $k-$th Euler number.
\end{lemma}

\begin{proof}
i) First of all notice that $\frac{1}{\lambda -1}\allowbreak =\allowbreak
\sum_{n\geq 1}1/\lambda ^{n},$ hence $S+\frac{1}{\lambda -1}\allowbreak
=\allowbreak \sum_{n\geq 1}\frac{1}{\lambda ^{n}}(X_{n}+1).$ Now since $%
P(X_{n}+1\allowbreak =0)\allowbreak =\allowbreak P(X_{n}+1\allowbreak
=\allowbreak 2)\allowbreak =\allowbreak 1/2$ we see that $\limfunc{supp}(S+%
\frac{1}{\lambda -1})\subset \lbrack 0,\frac{2}{\lambda -1}].$ Notice also
that if $\lambda \allowbreak =\allowbreak 2$ then $S+1\allowbreak
=\allowbreak 2\sum_{n\geq 1}\frac{1}{2^{n}}Y_{n},$ where $P(Y_{n}\allowbreak
=\allowbreak 0)\allowbreak =\allowbreak P(Y_{n}\allowbreak =\allowbreak
1)\allowbreak =\allowbreak 1/2.$ In other words $(S+1)/2$ is any number
chosen from $[0,1]$ with equal chances that is $(S+1)/2$ has uniform
distribution on $[0,1].$

When $\lambda \allowbreak =\allowbreak 3$ we see that $S+1/2$ is a number
that can be written with the help of $^{\prime }0^{\prime }$ and $^{\prime
}2^{\prime }$ in ternary expansion. In other words $S+1/2$ is number drawn
from Cantor set with equal chances. For $\lambda $ integer we argue in the
similar way.

ii) We have $\varphi _{S}(\lambda ^{k}t)\allowbreak =\allowbreak
\prod_{n\geq 1}\cos (\lambda ^{k}t\frac{1}{\lambda ^{n}})\allowbreak
=\allowbreak \varphi _{S}\left( t\right) \prod_{j=0}^{k-1}\cos \left(
\lambda ^{j}t\right) $.

iii) Fix integer $k$. Notice that we have:%
\begin{gather*}
S\left( \lambda \right) \allowbreak =\allowbreak \sum_{n\geq 1}\lambda
^{-n}X_{n}\allowbreak =\allowbreak \\
\sum_{j\geq 1}\lambda ^{-kj}X_{kj}\allowbreak +\allowbreak \newline
\sum_{j\geq 1}\lambda ^{-kj+1}X_{kj-1}\allowbreak +\allowbreak \ldots
\allowbreak +\allowbreak \sum_{j\geq 1}\lambda
^{-kj+k-1}X_{kj-k+1}\allowbreak = \\
\allowbreak \sum_{m=1}^{k}\lambda ^{m-1}\sum_{j\geq 1}\left( \lambda
^{k}\right) ^{-j}X_{kj-m+1}.
\end{gather*}
Now since by assumption all $X_{i}$ are i.i.d. we deduce that $S_{i}(\lambda
^{k})$ are i.i.d. random variables with distribution defined by $\varphi
_{\lambda ^{k}}(t).$ Hence we have $\varphi _{\lambda }(t)\allowbreak
=\allowbreak \prod_{j=1}^{k}\varphi _{\lambda ^{k}}(\lambda ^{j-1}t).$

iv) First of all we notice that $\varphi _{S}(t)$ is an even function hence
all derivatives of odd order at zero are equal $0$. Secondly let $\psi
_{\lambda }(t)$ denote moment generating function of $S(\lambda ).$ It is
easy to notice that $\psi _{\lambda }(s)\allowbreak =\allowbreak \varphi
_{S\left( \lambda \right) }(-is).$ Let us denote $m_{2n}\left( \lambda
\right) \allowbreak =\allowbreak \psi _{\lambda }^{(2n)}(0).$ Basing on the
elementary formula 
\begin{equation*}
\cosh (\alpha )\cosh (\beta )\allowbreak =\allowbreak \frac{1}{2}(\cosh
(\alpha +\beta )+\cosh \left( \alpha -\beta \right) ),
\end{equation*}%
we can easily obtain by induction the following identity: 
\begin{equation*}
\prod_{j=0}^{k-1}\cosh (\lambda ^{j}t)\allowbreak =\allowbreak \frac{1}{%
2^{k-1}}\sum_{i_{1}=0,\ldots ,i_{k-1}=0}^{1}\cosh
(t(1+\sum_{j=1}^{k}(2i_{j}-1)\lambda ^{j})).
\end{equation*}%
Since $\left. (\cosh \alpha t)^{(2n)}\right\vert _{t=0}\allowbreak
=\allowbreak \alpha ^{2n}$, we have 
\begin{gather*}
\left. \left( \prod_{j=0}^{k-1}\cosh (\lambda ^{j}t)\right)
^{(2n)}\allowbreak \right\vert _{t=0}\allowbreak = \\
\allowbreak \frac{1}{2^{k-1}}\sum_{i_{1}=0,\ldots ,i_{k-1}=0}^{1}\left.
\left( \cosh (t(1+\sum_{j=1}^{k}(2i_{j}-1)\lambda ^{j}))\right)
^{(2n)}\right\vert _{t=0}\allowbreak =W_{2n}^{(k)}\left( \lambda \right) .
\end{gather*}%
$\allowbreak $Now using Leibnitz formula for differentiation applied to (\ref%
{basic}) we get 
\begin{equation*}
f^{(2n)}(\lambda ^{k}t)\lambda ^{2kn}=\sum_{j=0}^{2n}\binom{2n}{j}%
(\prod_{i=0}^{k-1}\cosh \lambda ^{i}t)^{\left( j\right) }f^{\left(
2n-j\right) }\left( t\right) .
\end{equation*}%
Setting $t=\allowbreak 0$ and using the fact that all derivatives of both $f$
and $\cosh t$ of odd order at zero are zeros we get the desired formula.

v) We use result of \cite{Szab13} that states that for each $N$ inverse of
lower triangular matrix of degree $N\times N$ with $(i,j)$ entry $\binom{2i}{%
2j}$ is the lower triangular matrix with $(i,j)-th$ entry equal to $\binom{2i%
}{2j}E_{2(i-j)}.$
\end{proof}

\begin{remark}
Formula (\ref{bezp}) is known in a slightly different form it appeared in 
\cite{Lad92}, \cite{Hosk94} and \cite{Arnold11}. 
\end{remark}

\begin{remark}
Notice that polynomials $\left\{ W_{n}^{(k)}(\lambda )\right\} _{k,n\geq 1}$
satisfy the following recursive relationship for $k>1:$%
\begin{equation*}
W_{n}^{(k)}\left( \lambda \right) \allowbreak =\allowbreak \sum_{j=0}^{n}%
\binom{2n}{2j}\lambda ^{2j}W_{j}^{\left( k-1\right) }\left( \lambda \right) ,
\end{equation*}%
with $W_{n}^{(1)}(\lambda )\allowbreak =\allowbreak 1.$ Hence its generating
functions satisfy $\Theta _{k}(t)$ the following relationship 
\begin{equation*}
\Theta _{k}(t,\lambda )=\Theta _{k-1}(\lambda t,\lambda )\cosh t,
\end{equation*}%
where we have have denoted: $\Theta _{k}(t,\lambda )\allowbreak =\allowbreak
\sum_{n\geq 0}\frac{t^{2n}}{\left( 2n\right) !}W_{n}^{(k)}(\lambda ).$
\end{remark}

\begin{remark}
Notice that the above mentioned lemma provides an example of two singular
distributions whose convolution is a uniform distribution. Namely we have $%
S(4)\ast 2S(4)\allowbreak =\allowbreak S(2).$ Similarly we have $%
S(2)\allowbreak =\allowbreak S\left( 8\right) \ast 2S(8)\ast 4S(8)$ or $%
S(2)\allowbreak =\allowbreak S(2^{k})\ast \ldots \ast 2^{k-1}S(2^{k}).$ The
first example was already noticed by Kersher and Wintner in \cite%
{KershWint35},(22a).
\end{remark}

\begin{remark}
We can deduce even more from theses examples namely following the result of
Kersher \cite{Kersh36}, p.451 that characteristic functions $\varphi _{n}(t)$
of $S\left( n\right) $ (where $n$ is an integer $>2$) do not tend to zero as 
$t\longrightarrow \infty .$ Thus since we have $\varphi _{4}(t)\varphi
_{4}(2t)\allowbreak =\allowbreak \sin t/t$ and $\varphi _{8}(t)\varphi
_{8}(2t)\varphi _{8}(4t)\allowbreak =\allowbreak \sin t/t$ we deduce that if 
$t_{k}\longrightarrow \infty $ is a sequence such that $\left\vert \varphi
_{4}(t_{k})\right\vert >\varepsilon >0$ for suitable $\varepsilon $ then $%
\varphi _{4}(2t_{k})\longrightarrow 0.$ Similarly if $t_{k}\longrightarrow
\infty $ such that $\left\vert \varphi _{8}(t_{k})\right\vert >\varepsilon >0
$ then $\varphi _{8}(2t_{k})\varphi _{8}(4t_{k})\longrightarrow 0$. Similar
observations can be made can be made in more general situation. As it is
known from the papers of Erd\H{o}s \cite{Erdos39}, \cite{Erdos40}  the
situation that $\left\vert \varphi _{\lambda }(t_{k})\right\vert
>\varepsilon >0$ for some sequence $t_{k}\longrightarrow \infty $ occurs
when $\lambda $ is a Pisot number (briefly PV-number). On the other hand as
it is known roots of Pisot numbers are not Pisot, hence using above
mentioned result of Salem $\left\vert \varphi _{\lambda
^{1/k}}(t)\right\vert \longrightarrow 0$ as $t\longrightarrow \infty ,$
where $\lambda $ is some PV number and $k>1$ any integer. But we have $%
\varphi _{\lambda ^{1/k}}(t_{n})\allowbreak =\allowbreak
\prod_{j=1}^{k}\varphi _{\lambda }\left( \lambda ^{(j-1)/k}t_{n}\right)
\longrightarrow 0,$ where $t_{n}\longrightarrow \infty $ is such a sequence
that $\left\vert \varphi _{\lambda }(t_{n})\right\vert >\varepsilon >0.$
\end{remark}

\begin{remark}
One knows that if $\lambda \allowbreak =\allowbreak q/p$ where $p$ and $q$
are relatively prime integers and $p>1$ then $\varphi _{\lambda
}(t)\allowbreak =\allowbreak O((\log |t|)^{-\gamma })$ where $\gamma
\allowbreak =\allowbreak \gamma \left( p,q\right) >0$ as $t\longrightarrow
\infty $ (see \cite{Kersh36},(3)). Besides we know that then distribution of 
$S\left( \lambda \right) $ is singular. Hence from our considerations it
follows that if $\lambda \allowbreak =\allowbreak (q/p)^{1/k}$ for some
integer $k$ then $\varphi _{\lambda }(t)\allowbreak =\allowbreak O((\log
|t|)^{-k\gamma }).$ Is it also singular?
\end{remark}

\section{Moment sequences}

To give connection of certain moment sequences with some known integer
sequences let us remark the following:

\begin{remark}
i) 
\begin{eqnarray*}
9^{n}m_{2n}(3)\allowbreak &=&\allowbreak \allowbreak \sum_{j=0}^{n}\binom{2n%
}{2j}m_{2j}(3), \\
81^{n}m_{2n}(3) &=&\allowbreak \allowbreak \sum_{j=0}^{n}\binom{2n}{2j}%
m_{2j}(3)(2^{4(n-j)-1}+2^{2\left( n-j\right) -1}).
\end{eqnarray*}%
ii) 
\begin{eqnarray*}
5^{n}m_{2n}\left( \sqrt{5}\right) &=&\sum_{j=0}^{n}\binom{2n}{2j}m_{2j}(%
\sqrt{5}), \\
25^{n}m_{2n}\left( \sqrt{5}\right) &=&\sum_{j=0}^{n}\binom{2n}{2j}m_{2j}(%
\sqrt{5})4^{n-j}L_{2(n-j)}/2,
\end{eqnarray*}%
where $L_{n}$ denotes $n-$th Lucas number defined below.
\end{remark}

\begin{proof}
i) The first assertion is a direct application of (\ref{bezp}) while in
proving the second one we use (\ref{l^4}) and the fact that $\sum_{j=0}^{n}%
\binom{2n}{2j}9^{j}\allowbreak =\allowbreak 4^{n}(4^{n}+1)/2$ as shown by 
\cite{OEIS}, (seq. no. A026244). ii) Again the first statement follows (\ref%
{bezp}) while the second follows (\ref{l^4}) and the fact that $%
\sum_{j=0}^{n}\binom{2n}{2j}5^{j}\allowbreak =\allowbreak 4^{n}T_{n}(3/2),$
where $T_{n}$ denotes Chebyshev polynomial of the first kind. (\cite{OEIS},
seq. no. A099140). Further we use the fact that $T_{n}(3/2)\allowbreak
=\allowbreak L_{2n}/2.$ (\cite{OEIS}, seq. no. A005248).
\end{proof}

We also have the following Lemma.

\begin{lemma}
\label{_2}$\forall n\geq 1,k\geq 2:$%
\begin{equation}
m_{2n}(\lambda )\allowbreak =\allowbreak \sum_{i_{1},\ldots ,i_{k}=0}^{n}%
\frac{(2n)!}{(2i_{1})!\ldots (2i_{k})!}\lambda ^{2(i_{2}+2i_{3}\ldots
(k-1)i_{k})}\prod_{j=1}^{k}m_{2i_{j}}\left( \lambda ^{k}\right) .  \label{_K}
\end{equation}
In particular:%
\begin{eqnarray}
m_{2n}(\lambda )\allowbreak &=&\allowbreak \sum_{j=0}^{n}\binom{2n}{2j}%
\lambda ^{2j}m_{2j}(\lambda ^{2})m_{2n-2j}(\lambda ^{2}),  \label{kw} \\
m_{2n}(\lambda )\allowbreak &=&\allowbreak \sum_{\substack{ i,j=0  \\ %
i+j\leq n}}\frac{(2n)!}{(2i)!(2j)!(2(n-i-j))!}\times  \label{kub} \\
&&\lambda ^{2i}\lambda ^{4j}m_{2j}(\lambda ^{3})m_{2i}(\lambda
^{3})m_{2n-2i-2j}(\lambda ^{3}).  \notag
\end{eqnarray}
\end{lemma}

\begin{proof}
(\ref{_K}) follows directly Lemma \ref{_1}, iii).
\end{proof}

As a corollary we get the following four observations:

\begin{corollary}
\label{p22}i) $\forall n\geq 1:4^{n}\allowbreak =\allowbreak \sum_{j=0}^{n}%
\binom{2n+1}{2j+1}$ and $1\allowbreak =\allowbreak \sum_{j=0}^{n}\binom{2n+1%
}{2j+1}4^{j}E_{2(n-j)}.$

ii) $S\left( \sqrt{2}\right) \allowbreak $ has density%
\begin{equation*}
g(x)\allowbreak =\allowbreak \left\{ 
\begin{array}{ccc}
\sqrt{2}/4 & if & \left\vert x\right\vert \leq \sqrt{2}-1, \\ 
\sqrt{2}(\sqrt{2}+1-\left\vert x\right\vert )/8 & if & \sqrt{2}-1\leq
\left\vert x\right\vert \leq \sqrt{2}+1, \\ 
0 & if & \left\vert x\right\vert >1+\sqrt{2}.%
\end{array}%
\right.
\end{equation*}

iii) Let us denote $\delta _{n}\allowbreak =\allowbreak (\sqrt{2}+1)^{n}$,
then 
\begin{equation*}
m_{2n}(\sqrt{2})\allowbreak =\allowbreak \left( \delta _{2n+2}-\delta
_{2n+2}^{-1}\right) /(4\sqrt{2}(n+1)(2n+1)).
\end{equation*}
\end{corollary}

\begin{proof}
Since for $\lambda \allowbreak =\allowbreak 2$ random variable $S\allowbreak
\sim \allowbreak U[-1,1]$ its moments are equal to $ES^{2n}\allowbreak
=\allowbreak \frac{1}{2n+1}.$ Now we use Lemma \ref{_1} iii) and iv).

ii) From the proof of Lemma \ref{_2} it follows that $S\left( \sqrt{2}%
\right) \allowbreak \sim \allowbreak S\left( 2\right) \allowbreak
+\allowbreak \sqrt{2}S\left( 2\right) .$ Now keeping in mind that $S\left(
2\right) \allowbreak \sim \allowbreak U(-1,1)$ we deduce that $%
g(x)\allowbreak =\allowbreak \frac{\sqrt{2}}{8}\int_{-1}^{1}h(x-t)dt,$ where 
$h\left( x\right) \allowbreak =\allowbreak \left\{ 
\begin{array}{ccc}
\frac{\sqrt{2}}{4} & if & \left\vert x\right\vert \leq \sqrt{2}, \\ 
0 & if & otherwise.%
\end{array}%
\right. .$ iii) By straightforward calculations we get $m_{2n}(\sqrt{2}%
)\allowbreak =\allowbreak 2\int_{0}^{\sqrt{2}+1}x^{2n}g(x)dx\allowbreak
=\allowbreak \frac{\sqrt{2}}{2}\int_{0}^{\sqrt{2}-1}x^{2n}dx\allowbreak
+\allowbreak \frac{\sqrt{2}}{4}\int_{\sqrt{2}-1}^{\sqrt{2}+1}x^{2n}(\sqrt{2}%
+1-x)dx.$
\end{proof}

\begin{remark}
\label{rem1}Let us apply formulae: (\ref{kw}), (\ref{bezp}), (\ref{l^4}) and
observe by direct calculation that $2\sum_{j=0}^{n}\binom{2n}{2j}%
2^{j}\allowbreak =\allowbreak (1+\sqrt{2})^{2n}\allowbreak +\allowbreak (1-%
\sqrt{2})^{2n}$. We get the following identities: $\forall n\geq 1:$%
\begin{eqnarray*}
m_{2n}(\sqrt{2})\allowbreak  &=&\allowbreak \sum_{j=0}^{n}\binom{2n}{2j}%
\frac{2^{j}}{(2n-2j+1)(2j+1)} \\
&=&\frac{1}{2^{n}-1}\sum_{j=0}^{n-1}\binom{2n}{2j}m_{2j}(\sqrt{2}%
)\allowbreak  \\
&=&\frac{1}{4^{n}-1}\sum_{j=0}^{n-1}\binom{2n}{2j}m_{2j}(\sqrt{2})\tau
_{2(n-j)},
\end{eqnarray*}%
$\allowbreak \allowbreak $where $\tau _{n}\allowbreak =\allowbreak
(1+(-1)^{n})(\delta _{n}+\delta _{n}^{-1})/4.$
\end{remark}

Now let us recall the definition of the so called Pell and Pell--Lucas
numbers. Using sequence $\delta _{n}$ Pell numbers $\left\{ P_{n}\right\} $
and Pell--Lucas numbers $\left\{ Q_{n}\right\} $ are defined 
\begin{eqnarray}
P_{n}\allowbreak &=&\allowbreak (\delta _{n}+(-1)^{n+1}\delta _{n}^{-1})/(2%
\sqrt{2}),  \label{pell} \\
Q_{n}\allowbreak &=&\allowbreak \delta _{n}+\left( -1\right) ^{n}\delta
_{n}^{-1},  \label{P-L}
\end{eqnarray}%
where $\delta _{n}$ is defined in \ref{p22},iii).

Using these definitions we can rephrase assertions of Corollary \ref{p22}
and Remark \ref{rem1} adding to recently discovered (\cite{Sant06}, \cite%
{Benj08}) new identities satisfied by Pell and Pell-Lucas numbers and of
course probabilistic interpretation of Pell numbers.

\begin{remark}
\label{Pell}i) $m_{2n}(\sqrt{2})\allowbreak =\allowbreak \frac{P_{2n+2}}{%
(2n+2)(2n+1)},$ $\tau _{2n}\allowbreak =\allowbreak Q_{2n}/2.$

ii) $\forall n\geq 1:$%
\begin{eqnarray}
P_{2n+2}\allowbreak &=&\allowbreak \sum_{j=0}^{n}\binom{2n+2}{2j+1}2^{j},
\label{_pn} \\
Q_{2n} &=&2\sum_{j=0}^{n}\binom{2n}{2j}2^{j},  \label{_qn} \\
2^{n-1}P_{2n}\allowbreak &=&\allowbreak \sum_{j=0}^{n}\binom{2n}{2j}P_{2j},
\label{_spn} \\
2^{2n-1}P_{2n}\allowbreak &=&\sum_{j=0}^{n}\binom{2n}{2j}P_{2j}Q_{2(n-j)},
\label{_spqn} \\
\sum_{j=0}^{n}\binom{2n}{2j}(1+\sqrt{2})^{2j}\allowbreak &=&\allowbreak
2^{n-1}+2^{n-2}Q_{2n}+2^{n-1}\sqrt{2}P_{2n}.  \label{_ssilv}
\end{eqnarray}
\end{remark}

\begin{proof}
Only last statement requires justification. First we find that $%
\sum_{j=0}^{n}\binom{2n}{2j}Q_{2j}\allowbreak =\allowbreak 2^{n}(1+Q_{2n}/2)$
using (\ref{_qn}). Then we use (\ref{pell}), (\ref{P-L}) and (\ref{_spn}).
\end{proof}

\end{document}